 \documentclass[final,1p,times]{elsarticle}
\usepackage{listings}
\usepackage{color}
\usepackage[]{multicol}
\usepackage{amsthm}
\usepackage{caption}
\usepackage{here}
\usepackage{mathrsfs}
\usepackage{subcaption}
\usepackage{amsmath, amssymb}
\usepackage{ascmac}
\usepackage{url}
\usepackage{comment}
\usepackage{nccmath}
\usepackage {booktabs}
\def\vec#1{\mbox{\boldmath{$#1$}}}

\newtheorem{theorem}{Theorem}

\newtheorem{lemma}{Lemma}
\newtheorem{remark}{Remark}
 
\newtheorem{proposition}{Proposition}
\newtheorem{corollary}{Corollary}

\begin{document}

\begin{frontmatter}



\title{Chi-square approximation for the distribution of individual eigenvalues of a singular Wishart matrix }


\author[1]{Koki Shimizu \corref{mycorrespondingauthor}}
\author[1]{Hiroki Hashiguchi}

\cortext[mycorrespondingauthor]{{Corresponding author. Email address: \url{k-shimizu@rs.tus.ac.jp}~(K. Shimizu).}}

\affiliation[1]{organization={Tokyo University of Science},
            addressline={1-3 Kagurazaka}, 
            city={Shinjuku-ku},
            postcode={162-8601}, 
            state={Tokyo},
            country={Japan}}

\begin{abstract}
This paper discusses the approximate distributions of eigenvalues of a singular Wishart matrix.
We give the approximate joint density of eigenvalues by Laplace approximation for the hypergeometric functions of matrix arguments.
Furthermore, we show that the distribution of each eigenvalue can be approximated by the chi-square distribution with varying degrees of freedom when the population eigenvalues are infinitely dispersed. 
The derived result is applied to testing the equality of eigenvalues in two populations.
\end{abstract}


\end{frontmatter}


%
\section{Introduction}
\label{sec:01}
In multivariate analysis, some exact distributions of eigenvalues for a Wishart matrix are represented by hypergeometric functions of matrix arguments.
James (1964) classified multivariate statistics problems into five categories based on hypergeometric functions.
However, the convergence of these functions is slow, and their numerical computation is cumbersome when sample sizes or dimensions are large.
Consequently, the derivation of approximate distributions of eigenvalues has received a great deal of attention.
Sugiyama (1972) derived the approximate distribution for the largest eigenvalue through the integral representation of the confluent hypergeometric function.
Sugiura (1973) showed that the asymptotic distribution of the individual eigenvalues is expressed by a normal distribution for a large sample size.
The chi-square approximation was discussed when the population eigenvalues are infinitely dispersed in Takemura and Sheena (2005) and Kato and Hashiguchi (2014).
Approximations for hypergeometric functions have been developed and applied to the multivariate distribution theory in Butler and Wood (2002, 2005, 2022).
Butler and Wood (2002) provided the Laplace approximation for the hypergeometric functions of a single matrix argument. 
The numerical accuracies for that approximation were shown in the computation of noncentral moments of Wilk's lambda statistic and the likelihood ratio statistic for testing block independence. 
This approximation was extended for the case of two matrix arguments in Butler and Wood (2005).
All the results addressed above were carried out for eigenvalue distributions for a non-singular Wishart matrix.

Recently, the distribution of eigenvalues for the non-singular case has been extended to the singular case; see Shimizu and Hashiguchi (2021, 2022) and Shinozaki et al. (2022).
Shimizu and Hashiguchi~(2021) showed the exact distribution of the largest eigenvalue for a singular case is represented in terms of the confluent hypergeometric function as well as the non-singular case.
The generalized representation for the non-singular and singular cases under the elliptical model was provided by Shinozaki et al. (2022). 

The rest of this paper is organized as follows.
In Section~2, we apply the Laplace approximation introduced by Butler and Wood (2005) to the joint density of eigenvalues of a singular Wishart matrix. 
Furthermore, we show that the approximation for the distribution of the individual eigenvalues can be expressed by the chi-square distribution with varying degrees of freedom when the population covariance matrix has spiked eigenvalues. 
Section~3 discusses the equality of the individual eigenvalues in two populations.
Finally, we evaluate the precision of the chi-square approximation by comparing it to the empirical distribution through Monte Carlo simulation in Section 4.
\section{Approximate distributions of eigenvalues of a singular Wishart matrix}
\label{sec:02}
Suppose that an $m\times n$ real Gaussian random matrix $X$ is distributed as $X\sim$ $N_{m,n}(O,\Sigma \otimes I_n)$, 
where $O$ is the $m \times n$ zero matrix, $\Sigma$ is a $m\times m$ positive symmetric matrix, and $\otimes$ is the Kronecker product. 
This means that the column vectors of $X$ are independently and identically distributed (i.i.d.) from $N_{m}(0, \Sigma)$ with sample size $n$, where $0$ is the $m$-dimensional zero vector.
The eigenvalues of $\Sigma$ are denoted by $\lambda_1, \lambda_2, \dots, \lambda_m$ and $\lambda_1\geq \lambda_2 \geq \cdots \geq \lambda_m>0$. 
Subsequently, we define the singular Wishart matrix as $W=XX^\top$, where $m>n$ and its distribution is denoted by $W(n,\Sigma)$. 
The spectral decomposition of $W$ is represented as $W=H_1L_1H_1^\top$, where $L_1=\mathrm{diag}(\ell_1,\dots, \ell_n)$ with $\ell_1>\ell_2>\cdots>\ell_n>0$ and the $m\times n$ matrix $H_1$ is satisfied by $H_1 H_1^\top=I_n$. 
The set of all $m\times n$ matrices $H_1$ with orthonormal columns is called the Stiefel manifold, denoted by 
$V_{n, m}=\{H_1\mid  H_1H_1^\top=I_n \}$, where $m\geq n$.
The volume of $V_{n, m}$ is represented by
\begin{align*}
\mathrm{Vol}(V_{n, m})=\int_{H_1\in V_{n,m}}(H_1^\top dH_1)=\frac{2^n\pi^{mn/2}}{\Gamma_n(m/2)}.
\end{align*}
For the definition of the above exterior product $(H_1^\top dH_1)$, see page 63 of Muirhead (1982).
 If $m=n$, Stiefel manifold $V_{m, m}$ coincides with the orthogonal groups $O(m)$. 
Uhlig~(1994) gave the density of $W$ as
\begin{align*}
f(W)=\frac{\pi^{(-mn+n^2)/2}}{2^{mn/2}\Gamma_n(n/2)|\Sigma|^{n/2}}|L_1|^{(n-m-1)/2}\mathrm{etr}(-\Sigma^{-1} W/2),
\end{align*}
where $\Gamma_m(a)=\pi^{m(m-1)/4}\prod_{i=1}^{m}\Gamma\{a-(i-1)/2\}$ and $\mathrm{etr}(\cdot)=\exp(\mathrm{tr}(\cdot))$.
Srivastava~(2003) represented the joint density of eigenvalues of $W$ in a form that includes an integral over the Stiefel manifold;
 \begin{align}
 \label{joint-eigen}
 \nonumber
 f(\ell_1,\dots, \ell_n)=\mbox{}&\frac{2^{-nm/2}\pi^{n^2/2}}{|\Sigma|^{n/2}\Gamma_n(n/2)\Gamma_n(m/2)}\prod_{i=1}^n \ell_i^{(m-n-1)/2}\prod_{i<j}^{n}(\ell_i - \ell_j) \\
  &\times \int_{H_1\in V_{n, m}} \mathrm{etr}~\biggl(-\frac{1}{2}\Sigma^{-1}H_1L_1H_1^\top\biggl)~(dH_1),
\end{align} 
where $(dH_1)=\frac{(H_1^\top dH_1)}{\mathrm{Vol}(V_{n, m})}$ and $\int_{H_1 \in V_{n,m}}(dH_1)=1$.

 The above integral over the Steifel manifold was evaluated by Shimizu and Hashiguchi~(2021) as the hypergeometric functions of the matrix arguments.
We approximate \eqref{joint-eigen} by Laplace approximation for the hypergeometric functions of two matrix arguments provided in Butler and Wood (2005).
  
 For a positive integer $k$, let $\kappa=(\kappa_1,\kappa_2,\dots,\kappa_m)$ denote a partition of $k$ with $\kappa_1\geq \kappa_2 \geq \cdots \geq \kappa_m\geq 0$ and $\kappa_1+\cdots +\kappa_m=k$. The set of all partitions with less than or equal to $m$ is denoted by $P^k_{m}=\{\ \kappa=(\kappa_1,\dots,\kappa_m)\mid \kappa_1+\dots+\kappa_m=k, \kappa_1\geq \kappa_2 \geq\cdots \geq \kappa_m \geq 0 \}$. 
The  Pochhammer symbol for a partition $\kappa$ is defined as $(\alpha)_\kappa=\prod_{i=1}^{m}\{\alpha-(i-1)/2\}_{\kappa_i}$, where $(\alpha)_k=\alpha(\alpha+1)\cdots (\alpha+k-1)$ and $(\alpha)_0=1$. 
For integers, $p,q\geq 0$ and $m\times m$ real symmetric matrices $A$ and $B$, we define the hypergeometric function of two matrix arguments as
\begin{align}
\label{hypterpFq}
{}_pF_q{}^{(m)}(\mbox{\boldmath$\alpha$};\mbox{\boldmath$\beta$};A, B) = \sum_{k=0}^\infty \sum_{\kappa \in P^k_m}\frac{(\alpha_1)_{\kappa} \cdots (\alpha_p)}{(\beta_1)_{\kappa} \cdots (\beta_q)}\frac{C_\kappa(A)C_\kappa(B)}{k!C_\kappa(I_m)}, 
\end{align}
where $\mbox{\boldmath$\alpha$}=(\alpha_1,\ldots,\alpha_p)^\top$, $\mbox{\boldmath$\beta$}=(\beta_1,\ldots,\beta_q)^\top$ and $C_\kappa(A)$ is the zonal polynomial indexed by $\kappa$ with the symmetric matrix $A$, see the details given in Chapter~7 of Muirhead (1982).  
The hypergeometric functions with a single matrix are defined as 
\begin{align}
\label{hyptersingle;pFq}
{}_pF_q{}(\mbox{\boldmath$\alpha$};\mbox{\boldmath$\beta$};A)={}_pF_q{}(\mbox{\boldmath$\alpha$};\mbox{\boldmath$\beta$};A, I_m).  
\end{align}
The special cases ${}_1F_1{}$ and ${}_2F_1{}$ are called the confluent and Gauss hypergeometric functions, respectively.
Butler and Wood~(2002) proposed a Laplace approximation of ${}_1F_1{}$ and ${}_2F_1{}$ through their integral expressions. 
 They showed the accuracy of that approximation is greater than the previous results.
This approximation was extended to the complex case in Butler and Wood~(2022).
 The important property of \eqref{hypterpFq} is the integral representation over the orthogonal group 
\begin{align}
\label{integral-twohypter}
{}_pF_q{}^{(m)}(\mbox{\boldmath$\alpha$};\mbox{\boldmath$\beta$};A,B)
&=\int_{H\in O(m)}{_pF_q}(\mbox{\boldmath$\alpha$};\mbox{\boldmath$\beta$};AHBH^\top)(dH),
\end{align}
where $(dH)$ is the invariant measure on the $m\times m $ orthogonal group $O(m)$. 
Integral representations \eqref{integral-twohypter} are a useful tool for obtaining approximation of ${}_pF^{(m)}_q{}$. 
Asymptotic expansions of ${}_0F^{(m)}_0{}$ are given in Anderson~(1965) when both two positive definite matrix arguments are widely spaced.
Constantine and Muirhead (1976) gave the asymptotic behavior of ${}_0F^{(m)}_0{}$ when the population eigenvalues are multiple.  
From the integral expression \eqref{integral-twohypter}, Butler and Wood~(2005) provided Laplace approximations for ${}_pF_q{}^{(m)}$. 
 \begin{lemma}
 \label{laplace-approximation}
 Let the two diagonal matrices be $A=\mathrm{diag}(a_1,\ldots,a_m)$ and
$B=\mathrm{diag}(b_1,\ldots, b_1,b_2, \ldots, b_r, \ldots, b_r) $,
where $a_1 > a_2 > \cdots > a_m > 0$, $b_1 > b_2>\cdots > b_r \geq 0$, 
and $b_j$ have multiplicity $m_j$ in which $m=\sum_{j=1}^{r} m_{j}$.
Let $\Omega(m_1,\ldots,m_r) =\mathrm{Vol}(O(m))^{-1} \prod_{j=1}^{r} \mathrm{Vol}(O(m_j))$.
Then the Laplace approximation of ${}_pF_q{}^{(m)}$ is given as 
\begin{align*}
{}_p\hat{F}_q{}^{(m)}(\vec\alpha ; \vec\beta ; A , B ) = (2\pi)^{\frac{s}{2}}\Omega(m_1 ,\dots,m_r) J^{-\frac{1}{2}}
{}_pF_q(\vec{\alpha} ; \vec{\beta} ; AB ),
\end{align*}
where $s=\sum_{i=1}^{r-1}\sum_{j=i+1}^{r}m_im_j$ and Hessian $J$ is defined in Butler and Wood (2005).
 \end{lemma}
Shimizu and Hashiguchi~(2021) showed the following relationship
\begin{align}
\label{integral-formula}
\displaystyle \int_{H_1\in V_{n,m}}{_pF_q}(AH_1B_1H_1^\top)(dH_1)&=\displaystyle \int_{H\in O(m)}{_pF_q}(AHBH^\top) (dH)
 \end{align}
 for an $m \times m$ matrix $B=\left( 
\begin{array}{cc}
B_1 & O\\
O& O
\end{array}
\right)$, where $B_1$ is an $n \times n$ symmetric matrix and $O$ is the zero matrix.
 From \eqref{integral-formula}, the joint density \eqref{joint-eigen} can be rewritten by
 \begin{align}
 \label{integral-expression}
 \nonumber
 f(\ell_1,\dots, \ell_n)\propto\mbox{}& \int_{H\in O(m)} \mathrm{etr}~\biggl(-\frac{1}{2}\Sigma^{-1}HLH^\top\biggl)~(dH)\\
 =\mbox{}&{}_0F_0{}^{(m)}\biggl(-\frac{1}{2}\Sigma^{-1}, L\biggl),
 \end{align}
 where $L=\mathrm{diag}(\ell_1,\dots, \ell_n, 0,\dots,0)$ is the $m\times m$ matrix and the symbol ``$\propto$'' means that a constant required for scaling is removed. 
Applying Laplace's method to the above joint density, we have the approximation for the joint density of eigenvalues.
\begin{proposition}
\label{proposition-jointdensity}
The joint density of eigenvalues of a singular Wishart matrix by Laplace approximation is expressed by 
\begin{align}
\label{appro-jointdensity}
 \frac{\pi^{n(n-m)/2}}{2^{nm/2}|\Sigma|^{n/2}\Gamma_n(n/2)}\prod_{i=1}^{n}\ell_i^{(m-n-1)/2} \prod_{i<j}^{n} (\ell_i-\ell_j)~\mathrm{exp}\left(-\frac{1}{2}\sum_{i=1}^{n}\frac{\ell_i}{\lambda_i} \right)\prod_{i<j}^{n}\left(\frac{2\pi}{c_{ij}} \right)^{1/2}\prod_{i=1}^{n}\prod_{j=n+1}^{m}\left(\frac{2\pi}{d_{ij}} \right)^{1/2},
\end{align}
where $c_{ij}=\frac{(\ell_i-\ell_j)(\lambda_i-\lambda_j)}{\lambda_i \lambda_j} , d_{ij} = \frac{\ell_i(\lambda_i-\lambda_j)}{\lambda_i\lambda_j}$.
\begin{proof}
Applying Lemma~\ref{laplace-approximation} to the hypergeometric functions in \eqref{integral-expression}, the integral over the Stiefel manifold in \eqref{joint-eigen} is approximated by
\begin{align}
\label{appro}
\frac{2^n}{\mathrm {Vol}(V_{n, m})}\exp \left(-\frac{1}{2}\sum_{i=1}^n\frac{\ell_i}{\lambda_i}\right)\prod_{i<j}^n\left(\frac{2\pi}{c_{ij}}\right)^{1/2}\prod_{i=1}^n\prod_{j=n+1}^{m}\left(\frac{2\pi}{d_{ij}}\right)^{1/2}.
\end{align}
Substituting \eqref{appro} to \eqref{joint-eigen}, we have the desired result.  
\end{proof}
\end{proposition}

In order to derive the approximate distributions of individual eigenvalues, we define the spiked covariance model $\rho_k$ that implies the first $k$-th eigenvalues of $\Sigma>0$ are infinitely dispersed, namely
\begin{align}
\label{rho-eigen}
\rho_k=\mathrm{max}~\biggl( \frac{\lambda_{2}}{\lambda_{1}}, \frac{\lambda_{3}}{\lambda_{2}}, \dots, \frac{\lambda_{k+1}}{\lambda_{k}}\biggl)~\to 0
\end{align}
for $k=1,\dots, n$.
Under the condition of \eqref{rho-eigen}, Takemura and Sheena~(2005) proved that the distribution of individual eigenvalues for a non-singular Wishart matrix is approximated by a chi-square distribution.
The improvement for that approximation, that is, when the condition listed in \eqref{rho-eigen} cannot be assumed, was discussed in Tsukada and Sugiyama~(2021).
The following lemma was provided by Takemura and Sheena (2005) and Nasuda~et al. (2022) in the non-singular case and could be easily extended to the singular case.
\begin{lemma}
\label{lemma:dispersed}
Let $W \sim W_m(n, \Sigma)$, where $m>n$ and $\ell_1, \ell_2, \dots, \ell_n$ be the eigenvalues of $W$.
If $\rho_k \rightarrow 0$, we have 
$$
r_k=\mathrm{max}~\biggl( \frac{\ell_{2}}{\ell_{1}}, \frac{\ell_{3}}{\lambda_{2}}, \dots, \frac{\ell_{k+1}}{\ell_{k}}\biggl)~\overset{p}{\to}0, ~(k=1,\dots, n)$$
in the sense that $ \forall \epsilon>0$, $ \exists \delta>0$,
\begin{align*}
\rho_k<\delta \rightarrow \mathrm{Pr}~(r_k>\epsilon) <\epsilon.
\end{align*}
\end{lemma}
From Proposition~\ref{proposition-jointdensity} and Lemma~\ref{lemma:dispersed}, we obtain the chi-square approximation that is the main result of this paper.
\begin{theorem}
\label{chi-dist}
Let $W \sim W_m(n, \Sigma)$, where $m>n$ and $\ell_1, \ell_2, \dots, \ell_n$ be the eigenvalues of $W$. If $\rho_n \rightarrow 0$, it holds that 
$$
\ell_i/\lambda_i \stackrel{d}{\rightarrow} \chi^2_{n-i+1},~~1\leq i \leq n,
$$   
where $\chi^2$ is a chi-square distribution with $n-i+1$ degrees of freedom and the symbol ``$\stackrel{d}{\rightarrow}$" means convergence in the distribution.
\begin{proof}
First, we rewrite the approximate distribution \eqref{appro-jointdensity} as
\begin{align*}
f(\ell_1,\dots,\ell_n)=\mbox{}&\frac{1}{2^{n(n+1)/4}|\Sigma|^{n/2}\prod_{i=1}^{n}\Gamma(\frac{n-i+1}{2})}\prod_{i=1}^{n}\ell_i^{(m-n-1)/2}  \mathrm{exp}\left(-\frac{1}{2}\sum_{i=1}^{n}\frac{\ell_i}{\lambda_i} \right)\\
&\times \prod_{i<j}^{n}\left\{ (\ell_i-\ell_j)\left(\frac{\lambda_i\lambda_j}{\lambda_i-\lambda_j}\right) \right \}^{1/2}\prod_{i=1}^{n}\prod_{j=n+1}^{m}\left(\frac{1}{d_{ij}} \right)^{1/2}.
\end{align*}
From Lemma~\ref{lemma:dispersed}, we have
\begin{align*}
\prod_{i<j}^{n}(\ell_i-\ell_j)^{1/2} &=\prod_{i<j}^{n}\ell_i^{1/2}\biggl(1-\frac{\ell_j}{\ell_i}\biggl)^{1/2} \approx \prod_{i=1}^{n}\ell_i^{(n-i)/2},\\
\prod_{i<j}^{n}\left(\frac{\lambda_i\lambda_j}{\lambda_i-\lambda_j}\right)^{1/2}&=\prod_{i<j}^{n}\biggl(\frac{\lambda_j}{1-\lambda_i/\lambda_j}\biggl)^{1/2} \approx \prod_{i=1}^{n}\lambda_i^{(i-1)/2}.
\end{align*}
Furthermore, we note $\displaystyle|\Sigma|^{n/2}=\prod_{i=1}^{n}\lambda_i^{n/2}\prod_{i=1}^{n}\prod_{j=n+1}^{m}\lambda_j^{1/2}$, the joint density \eqref{appro-jointdensity} can be written as 
\begin{align*}
&f(\ell_1,\dots,\ell_n)\\
\approx &\prod_{i=1}^{n}\left\{ \frac{\ell_i^{(n-i-1)/2}}{(2\lambda_i)^{(n-i+1)/2}\Gamma(\frac{n-i+1}{2})}\mathrm{exp}\left( -\frac{\ell_i}{2\lambda_i}\right) \right\}\
   \prod_{i=1}^{n}\ell_i^{(m-n)/2} \prod_{i=1}^{n}\prod_{j=n+1}^{m} \biggl(\frac{1}{\lambda_j}\biggl)^{1/2}\prod_{i=1}^{n}\prod_{j=n+1}^{m}\left(\frac{1}{d_{ij}} \right)^{1/2}\\
   =&\prod_{i=1}^{n}\left\{ \frac{\ell_i^{(n-i-1)/2}}{(2\lambda_i)^{(n-i+1)/2}\Gamma(\frac{n-i+1}{2})}\mathrm{exp}\left( \frac{\ell_i}{2\lambda_i}\right) \right\}\
  \prod_{i=1}^{n}\prod_{j=n+1}^{m}\ell_i^{\frac{1}{2}} \prod_{i=1}^{n}\prod_{i=j+1}^{m} \biggl(\frac{1}{\lambda_j}\biggl)^{1/2}\prod_{i=1}^{n}\prod_{j=n+1}^{m}\left(\frac{1}{d_{ij}} \right)^{1/2}\\
    =&\prod_{i=1}^{n}\left\{ \frac{\ell_i^{(n-i-1)/2}}{(2\lambda_i)^{(n-i+1)/2}\Gamma(\frac{n-i+1}{2})}\mathrm{exp}\left( -\frac{\ell_i}{2\lambda_i}\right) \right\} \\
  =&\prod_{i=1}^{n}g_{n-i+1}(\ell_i/\lambda_i),
   \end{align*}
   where $g_{n-i+1}(\cdot)$ is the density function of the chi-square distribution with degree of freedom $n-i+1$.
   \end{proof}
   \end{theorem}
   \begin{remark}
Note that if only the first $k$ population eigenvalues are infinitely dispersed, that is $\rho_k \to 0$, $\ell_k/\lambda_k$ can be approximated by $g_{n-k+1}$ similar to the proof in Theorem~\ref{chi-dist}.
   \end{remark}
   In the context of the High Dimension-Low Sample Size (HDLSS) setting, the asymptotic behavior of the eigenvalue distribution of a sample covariance matrix was discussed in Ahn et al. (2007), Jung and Marron (2009), and Bolivar-Cime and Perez-Abreu (2014).
 Jung and Marron (2009) showed that the spiked sample eigenvalues are approximated by the chi-square distribution with a degree of freedom $n$.
In contrast, Theorem \ref{chi-dist} gives the approximation of the distribution of $\ell_k$ by a chi-square distribution with varying degrees of freedom.

\section{Application to test for equality of the individual eigenvalues}
\label{sec:03}
This section discusses testing for equality of individual eigenvalues of the covariance matrix in two populations. 
For testing problems, we give the approximate distribution of the statistic based on the derived results from the previous section.

Let an $m\times n_i$ Gaussian random matrix $X^{(i)}$ be distributed as $X^{(i)}\sim N_{m, n}(O, \Sigma^{(i)} \otimes I_{n_i})$, where $\Sigma^{(i)}>0$ and $i=1, 2$.
The eigenvalues of $\Sigma^{(i)}$ are denoted by $\lambda_1^{(i)}, \lambda_2^{(i)}, \dots, \lambda_m^{(i)}$, where $\lambda_1^{(i)}\geq \lambda_2^{(i)} \geq \cdots \geq \lambda_m^{(i)}>0$.
We denote the eigenvalues of $W^{(i)}=X^{(i)}{X^{(i)}}^\top$ by $\ell_1^{(i)}, \ell_2^{(i)}, \dots, \ell_m^{(i)}$, where $\ell_1^{(i)}> \ell_2^{(i)} > \cdots > \ell_m^{(i)}\geq 0$.
For fixed $k$, we consider the test of the equality of the individual eigenvalues in two populations as  
\begin{align}
 \label{eigenvalue-test}
H_0:\lambda^{(1)}_k=\lambda^{(2)}_k, \text{vs.}~H_1:\lambda^{(1)}_k \neq \lambda^{(2)}_k.
\end{align}
Sugiyama and Ushizawa (1998) reduced \eqref{eigenvalue-test} to the equality of variance test for the principal components and proposed a testing procedure using the Ansari-Bradley test.
Takeda~(2001) proposed the test statistic $\ell_k^{(1)}/\ell_k^{(2)}$ with $n\geq m$ for (\ref{eigenvalue-test}) and derived the exact distribution of $\ell_1^{(1)}/\ell_1^{(2)}$. 
Since Johnstone (2001) indicated that the first few eigenvalues are very large compared to the others in the large dimensional setting, it is essential to understand how the distribution for the first few eigenvalues is constructed. 
We provide the exact density function of $\ell_1^{(1)}/\ell_1^{(2)}$ with $ n< m$ in the same way as Takeda~(2001).
\begin{theorem}
 Let $W^{(1)}$ and $W^{(2)}$ be two independent Wishart matrices with distribution $W_m(n_1,\Sigma^{(1)})$ and $W_m(n_2,\Sigma^{(2)})$, respectively, where $m>n_i~(i=1,2)$.
Then we have the density of $q=\ell_1^{(1)}/\ell_1^{(2)}$ as
\begin{align}
\label{ratioell1}
\nonumber
f(q)=\vbox{}&C~\sum_{k=0}^{\infty}\sum_ {\kappa\in P_{m}^{k}}\sum_{t=0}^{\infty}\sum_ {\tau\in P_{m}^{t}}\frac{\{(m+1)/2\}_{\kappa}C_{\kappa}({\Sigma^{(1)}}^{-1}/2)}{\{(n_{1}+m+1)/2\}_{\kappa}k!}
\frac{\{(m+1)/2\}_{\tau}C_{\tau}({\Sigma^{(2)}}^{-1}/2)}{\{(n_{2}+m+1)/2\}_{\tau}t!}\\ \nonumber
&\times \biggl\{(mn_1/2+k)(mn_2/2+t)q^{mn_2/2+t-1}\Gamma(u)/v^u\\ \nonumber
&~~-(mn_1/2+k)(\mathrm{tr}{\Sigma^{(2)}}^{-1}/2)q^{mn_2/2+t}\Gamma(u+1)/v^{u+1}\\ \nonumber
&~~-(mn_2/2+t)(\mathrm{tr}{\Sigma^{(1)}}^{-1}/2)q^{mn_2/2+t-1}\Gamma(u+1)/v^{u+1}\\ 
&~~+(\mathrm{tr}{\Sigma^{(1)}}^{-1}/2)(\mathrm{tr}{\Sigma^{(2)}}^{-1}/2)q^{mn_2/2+t}\Gamma(u+2)/v^{u+2}\biggl\}, 
\end{align}
where $u=m(n_1+n_2)/2+k+t$, $v=\mathrm{tr}{\Sigma^{(1)}}^{-1}-q\mathrm{tr}{\Sigma^{(2)}}^{-1}$ and 
$$
C=\frac{\Gamma_{n_1}\{(n_1+1)/2\}\Gamma_{n_2}\{(n_2+1)/2\}}{2^{m(n_1+n_2)/2}\Gamma_{n_1}\{(n_1+m+1)/2)\Gamma_{n_2}\{(n_2+m+1)/2)\}|\Sigma|^{n_1/2}|\Sigma|^{n_2/2}}.
$$
\end{theorem}
\begin{proof}
The exact expression of $\ell_1^{(i)}$ was given by Shimizu and Hashiguchi~(2021) as
\begin{align}
\label{ell1}
\displaystyle\mathrm{Pr}(\ell_1^{(i)}<x) = \frac{
\Gamma_{n_i}(\frac{n_i+1}{2})(\frac{x}{2})^{ mn_i/2}}{\Gamma_{n_i}(\frac{n_i+m+1}{2})|\Sigma^{(i)}|^{n_i/2}} \mathrm{exp}\biggl(-\frac{x}{2}\mathrm{tr}{\Sigma^{(i)}}^{-1}\biggl)~~
{}_1F_1{}\left(\frac{m+1}{2};\frac{n_i+m+1}{2};\frac{x}{2}{\Sigma^{(i)}}^{-1}\right ).
\end{align}
The derivative of (\ref{ell1}) is represented by 
\begin{align}
\label{densityell1}
\nonumber
f(x)=\mbox{}&\frac{\Gamma_{n_i}\{(n_i+1)/2\}}{2^{mn_i/2}\Gamma_{n_i}\{(n_i+m+1)/2)\}|\Sigma^{(i)}|^{n_i/2}}\sum_{k=0}^{\infty}\sum_ {\kappa\in P_{m}^{k}}\frac{\{(m+1)/2\}_{\kappa}C_{\kappa}({\Sigma^{(i)}}^{-1}/2)}{\{(n_{i}+m+1)/2\}_{\kappa}k!}\\
& \times\mathrm{exp}\biggl(-\frac{x}{2}\mathrm{tr}{\Sigma^{(i)}}^{-1}\biggl)\biggl\{(n_{i}m/2+k)x^{mn_{i}/2+k-1}-(\mathrm{tr}{\Sigma^{(i)}}^{-1}/2)x^{mn_{i}/2+k}\biggl\}.
\end{align}
From (\ref{densityell1}), we have the joint density of $\ell_1^{(1)}$ and $\ell_1^{(2)}$ as
\begin{align*}
f(x, y)=\vbox{}&C~\sum_{k=0}^{\infty}\sum_ {\kappa\in P_{m}^{k}}\sum_{t=0}^{\infty}\sum_ {\tau\in P_{m}^{t}}\frac{\{(m+1)/2\}_{\kappa}C_{\kappa}({\Sigma^{(1)}}^{-1}/2)}{\{(n_{1}+m+1)/2\}_{\kappa}k!}
\frac{\{(m+1)/2\}_{\tau}C_{\tau}({\Sigma^{(2)}}^{-1}/2)}{\{(n_{2}+m+1)/2\}_{\tau}t!}\\
&\times\biggl\{(mn_{1}/2+k)x^{mn_{1}/2+k-1}-(\mathrm{tr}{\Sigma^{(1)}}^{-1}/2)x^{mn_{1}/2+k}\biggl\}\\
&\times \biggl\{(mn_{2}/2+t)y^{mn_{2}/2+t-1}-(\mathrm{tr}{\Sigma^{(2)}}^{-1}/2)y^{mn_{2}/2+t}\biggl\}\mathrm{exp}\biggl(-\frac{x}{2}\mathrm{tr}{\Sigma^{(1)}}^{-1}\biggl)~\mathrm{exp}\biggl(-\frac{y}{2}\mathrm{tr}{\Sigma^{(2)}}^{-1}\biggl).
\end{align*}
Translating $x$ and $y$ to $q=y/x$ and $r=x$, we have
\begin{align*}
f(q, r)=\vbox{}&C~\sum_{k=0}^{\infty}\sum_ {\kappa\in P_{m}^{k}}\sum_{t=0}^{\infty}\sum_ {\tau\in P_{m}^{t}}\frac{\{(m+1)/2\}_{\kappa}C_{\kappa}({\Sigma^{(1)}}^{-1}/2)}{\{(n_{1}+m+1)/2\}_{\kappa}k!}
 \frac{\{(m+1)/2\}_{\tau}C_{\tau}({\Sigma^{(2)}}^{-1}/2)}{\{(n_{2}+m+1)/2\}_{\tau}t!}\\
&\times\biggl\{(mn_1/2+k)(mn_2/2+t)q^{mn_2/2+t-1}r^{m(n_1+n_2)/2+k+t-1}\\
&~~~~~ -(mn_1/2+k)(\mathrm{tr}{\Sigma^{(2)}}^{-1}/2)q^{mn_2/2+t}r^{m(n_1+n_2)/2+k+t}\\
&~~~~~-(mn_2/2+t)(\mathrm{tr}{\Sigma^{(1)}}^{-1}/2)q^{mn_2/2+t-1}r^{m(n_1+n_2)/2+k+t}\\
&~~~~~+(\mathrm{tr}{\Sigma^{(1)}}^{-1}/2)(\mathrm{tr}{\Sigma^{(2)}}^{-1}/2)q^{mn_2/2+t}r^{m(n_1+n_2)/2+k+t+1}\biggl\}\\
&~~~~~\times \mathrm{exp}\biggl(-(\mathrm{tr}{\Sigma^{(1)}}^{-1}-q\mathrm{tr}{\Sigma^{(2)}}^{-1})r\biggl).
\end{align*}
Noting that $\int_{0}^{\infty} x^{\alpha-1}e^{-\beta x}dx = \Gamma(\alpha)/\beta^\alpha$, where $\alpha, \beta>0$, and integrating $r$ with respecet to $f(q, r)$, we have the desired result. 
\end{proof}
As the dimension increases, it is difficult to perform the numerical computation of \eqref{ratioell1} due to the high computational complexity.
 From Theorem~\ref{chi-dist}, we give the approximate distribution for \eqref{ratioell1} by $F$-distribution.
\begin{corollary}
\label{F-appro}
Let $W^{(1)}$ and $W^{(2)}$ be two independent Wishart matrices with distribution $W_m(n_1,\Sigma^{(1)})$ and $W_m(n_2,\Sigma^{(2)})$, respectively, where $m>n_i~(i=1,2)$ and $\ell_1^{(i)}, \ell_2^{(i)}, \dots, \ell_n^{(i)}$ are the eigenvalues of $W^{(i)}$. 
If the first $k$-th eigenvalues of $\Sigma^{(i)}$ are spiked, then we have
\begin{align*}
\frac{\ell_{k}^{(1)}/\{(n_1-k+1)\lambda_{n_1-k+1}^{(1)}\}}{\ell_{k}^{(2)}/\{(n_2-k+1)\lambda_{n_2-k+1}^{(2)}\}} \stackrel{d}{\rightarrow} F_{(n_1-k+1, n_2-k+1)}, 
\end{align*}
where $F$ is an F distribution with $n_1$ and $n_2$ degrees of freedom.
\end{corollary}

\section{Simulation study}
We investigate the accuracy of the approximation for the derived distributions.
In the simulation study, we consider the following population covariance matrix;
\begin{align}
\label{spikedcov}
\Sigma=\mathrm{diag}(a^{b}, a^{b/2},\dots, a^{b/m}),
\end{align}
where $a,b>0$.
In the large-dimensional setting, mainly the accuracy of the approximate distributions for the largest and second eigenvalues was investigated; see Iimori et al. (2013) and Sugiyama~et al. (2013).
We set $(a, b)=(200, 3)$ as Case~1 and $(a, b)=(50, 3)$ as Case~2.
These two cases imply that the population covariance matrix has two spiked eigenvalues.
Parameter $\rho_k$ in \eqref{rho-eigen} is smaller in Case~1 than in Case~2.
We denote ${{F_{1}}}(x)$ and ${{F_{2}}}(x)$ as the chi-square distributions with $n$ and $n-1$ degrees of freedom, which are the approximate distributions of the largest and second eigenvalues, respectively.
The empirical distribution based on $10^6$ Monte Carlo simulations is denoted by ${{F_\mathrm{sim}}}$.
Tables \ref{table1} and \ref{table2} shows the $\alpha$-percentile points of the distributions of $\ell_1$ and $\ell_2$ for $m=50$ and $n=10$, respectively.
 From the simulation study, we know that sufficient accuracy of approximation for the largest eigenvalue has already been obtained in Case~2.
Case 1 is more accurate than Case 2 for the second eigenvalue.
It is seen that the desired accuracy can be achieved when the parameter $\rho_k$ is small. 

%
\begin{table}[H]
\caption{Percentile points of the distributions of $\ell_1$ and $\ell_2$ of $W_{50}(10,\Sigma)$ (Case~1)}  \label{table1}
\begin{center}
\begin{tabular}{c}
       \captionsetup{labelformat=empty,labelsep=none}
{\begin{tabular}{@{}cccccccccc@{}} \toprule
$\alpha$&${{F^{-1}_\mathrm{sim}}}(\alpha)$&${{F_{1}^{-1}}}(\alpha)$& & $\alpha$ &${{F^{-1}_\mathrm{sim}}}(\alpha)$& ${{F_{2}^{-1}}}(\alpha)$\\ \toprule
0.99&23.2359&23.2093& &0.99&21.791&21.666\\
0.95&18.3026 &18.307& &0.95&17.0601&16.919\\
0.90&15.9825&15.9872 & &0.90&14.8377&14.6837\\
0.50&9.34466 & 9.34182& &0.50&8.48676&8.34283\\
0.05&3.94389&3.9403& &0.05&3.47796&3.32511\\ \toprule
\end{tabular}}
  \end {tabular}
  \end{center}
\end{table}
\begin{table}[H]
\caption{Percentile points of the distributions of $\ell_1$ and $\ell_2$ of $W_{50}(10,\Sigma)$ (Case~2)}   \label{table2}
\begin{center}
\begin{tabular}{c}
       \captionsetup{labelformat=empty,labelsep=none}
{\begin{tabular}{@{}cccccccccc@{}} \toprule
$\alpha$&${{F^{-1}_\mathrm{sim}}}(\alpha)$&${{F_{1}^{-1}}}(\alpha)$& & $\alpha$ &${{F^{-1}_\mathrm{sim}}}(\alpha)$& ${{F_{2}^{-1}}}(\alpha)$\\ \toprule
0.99&23.239&23.2093&&0.99&22.1285&21.666\\
0.95&18.306 &18.307&&0.95&17.4079&16.919\\
0.90&15.9857&15.9872 &&0.90&15.1856 &14.6837\\
0.50&9.34844 & 9.34182&&0.50&8.84394&8.34283\\
0.05&3.94744&3.9403&&0.05&3.86566 &3.32511\\ \toprule
\end{tabular}}
  \end {tabular}
  \end{center}
\end{table}
Tables~\ref{table3} and \ref{table4} present the chi-square probabilities for Case~1 in the upper percentile points from the empirical distribution.  
In this simulation study, we set $m = 20, 30, 40, 100$ and $n = 5,15$.
We can observe that all probabilities are close to $\alpha$.
\begin{table}[H]
\caption{Approximate probabilities  of $\ell_1$ based on the empirical percentile points (Case~1)}  \label{table3}
\begin{center}
\begin{tabular}{c}
       \captionsetup{labelformat=empty,labelsep=none}
{\begin{tabular}{@{}ccccc@{}} \toprule
$n$&$m$&\multicolumn{3}{c}{$\alpha$}  \\ \toprule
&&0.90& 0.95&0.99\\ \toprule 
$5$&20& 0.900047&0.950011&0.990165\\ 
&30&0.900227&0.950002&0.990018\\  
&40&0.900173&0.950185&0.990072\\ 
&100&0.899952&0.949874&0.990048\\    \toprule
15&20&0.900409&0.950375&0.990088\\   
&30&0.900258&0.950232&0.990124\\  
&40&0.900729&0.950464&0.990213\\  
&100&0.900331&0.950039&0.990083\\  
\noalign{\smallskip}\toprule
\end{tabular}}
  \end {tabular}
  \end{center}
\end{table}

\begin{table}[H]
\caption{Approximate probabilities  of $\ell_2$ based on the empirical percentile points (Case~1)}  \label{table4}
\begin{center}
\begin{tabular}{c}
       \captionsetup{labelformat=empty,labelsep=none}
{\begin{tabular}{@{}ccccc@{}} \toprule
$n$&$m$&\multicolumn{3}{c}{$\alpha$}  \\ \toprule
&&0.90& 0.95&0.99\\ \toprule 
$5$&20& 0.904868&0.952295&0.990482\\ 
&30& 0.905512&0.953016&0.990761\\  
&40&0.905838 &0.952924&0.990609\\ 
&100& 0.906393&0.95315&0.990613\\    \toprule
15&20& 0.903118&0.951587&0.990331\\   
&30& 0.903275&0.951856&0.990335\\  
&40& 0.903943&0.952062&0.990315\\  
&100& 0.904027&0.952163&0.990498\\  
\noalign{\smallskip}\toprule
\end{tabular}}
  \end {tabular}
  \end{center}
\end{table}

Finally, we give the graph of the density of $F$ distribution in Corollary~\ref{F-appro} compared to the empirical distribution function.
In Fig~\ref{fig1},  we superimpose the graph of $F$ approximation with the histogram of $\ell_1^{(1)}/\ell_1^{(2)}$ for $n_i=10~(i=1,2)$ and $m=30$ in Case~2. 
The vertical line and histograms show the empirical distribution of the $\ell^{(1)}_1/\ell^{(2)}_1$ based on $10^6$ iteration, respectively.
The solid line is the density function of the $F$ distribution. 
From the $95\%$ points of $F_{\mathrm{sim}}$, we can confirm that the approximate probability is 0.950.
\begin{figure}[H] \label{fig1}
\begin{center}
\includegraphics[width=7cm]{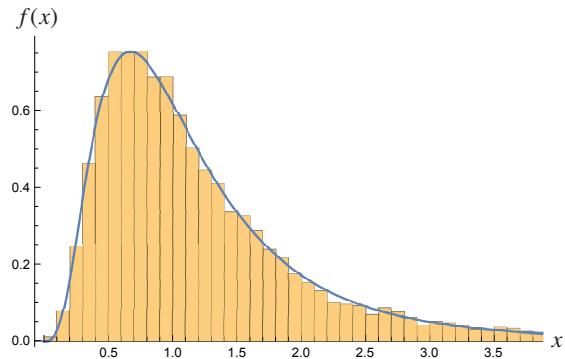}
\rlap{\raisebox{28.0ex}{\kern-20.em{\small $f(x)$}}}%
\rlap{\raisebox{.15cm}{\kern0cm{\small $x$}}}
\caption{$n_i=10~(i=1,2)$ and $m=30$ }
 \label {fig1}
\end{center}
\end{figure}

\section*{Acknowlegments}
The first author has received partial funding from Grant-in-Aid for JSPS Fellows (No.22KJ2804).





\end{document}